\documentclass{amsart}
\usepackage{cases}
\usepackage{amsfonts}
\usepackage{amssymb}
\usepackage{mathrsfs}

\newtheorem{theorem}{Theorem}[section]
\newtheorem{corollary}[theorem]{Corollary}
\newtheorem{lemma}[theorem]{Lemma}

\theoremstyle{definition}
\newtheorem{remark}[theorem]{Remark}

\numberwithin{equation}{section}

\newcommand{\Cn}{\mathbb{C}^n}
\newcommand{\C}{\mathbb{C}}

\newcommand{\Z}{\mathbb{Z}}

\DeclareMathOperator{\IDA}{IDA}
\DeclareMathOperator{\VDA}{VDA}
\DeclareMathOperator{\BDA}{BDA}
\DeclareMathOperator{\BMO}{BMO}

\DeclareMathOperator{\IMO}{IMO}
\DeclareMathOperator{\rank}{rank}
\DeclareMathOperator{\linspan}{span}

\begin{document}

\author{Zhangjian Hu}
\address{Department of Mathematics, Huzhou University, Huzhou, Zhejiang, China}
\email{huzj@zjhu.edu.cn}

\author{Jani A. Virtanen}
\address{Department of Mathematics and Statistics, University of Reading, Reading, England}
\email{j.a.virtanen@reading.ac.uk}
\address{Department of Mathematics and Statistics, University of Helsinki, Helsinki, Finland}
\email{jani.virtanen@helsinki.fi}

\title[Schatten class Hankel operators]{Schatten class Hankel operators on the Segal-Bargmann space and the Berger-Coburn phenomenon}



\thanks{Z.~Hu was supported in part by the National Natural Science Foundation of China (12071130, 12171150) and J.~Virtanen was supported in part by Engineering and Physical Sciences Research Council grant EP/T008636/1.}

\begin{abstract}
We give a complete characterization of Schatten class Hankel operators $H_f$ acting on weighted Segal-Bargmann spaces $F^2(\varphi)$ using the notion of integral distance to analytic functions in $\mathbb{C}^n$ and H\"ormander's $\bar\partial$-theory. Using our characterization, for $f\in L^\infty$ and $1<p<\infty$, we prove that $H_f$ is in the Schatten class $S_p$ if and only if $H_{\bar{f}}\in S_p$, which was previously known only for the Hilbert-Schmidt class $S_2$ of the standard Segal-Bargmann space $F^2(\varphi)$ with $\varphi(z) = \alpha |z|^2$.
\end{abstract}

\maketitle

\section{Introduction and main results}

We denote by $F^2$ the (classical) Segal-Bargmann space of Gaussian square-integrable entire functions on $\C^n$ and let $P$ be the orthogonal projection of $L^2$ onto $F^2$. For a bounded function $f$, the Hankel operator
$$
	H_f = (I-P)M_f
$$
is a bounded linear operator on $F^2$, where $M_fg = fg$ is the multiplication operator and $I$ is the identity operator. In addition to $F^2$, which is a central subject in quantum physics, we also consider Hankel operators on weighted Segal-Bargmann spaces. Besides intrinsic interest in operator theory,  the treatment of weighted spaces may be useful to understand more complicated quantum phenomena.

A unique feature in the theory of the Segal-Bargmann space is the property that the Hankel operator $H_f$ is compact if and only if $H_{\bar f}$ is compact when $f\in L^\infty$. This result of Berger and Coburn~\cite{BC1987} is not true for Hankel operators on the Bergman space or the Hardy space. In~\cite{Zh12}, Zhu conjectured that a possible explanation for this difference is the lack of bounded analytic functions on the entire complex plane, which was indeed confirmed recently in~\cite{HaVi20}. A natural question then arises as to whether an analogous phenomenon holds true for Hankel operators in the Schatten classes $S_p$ (defined in Section~\ref{Schatten classes}). This was answered in the affirmative for the Hilbert-Schmidt class $S_2$ by Bauer~\cite{Bauer} while Xia and Zheng~\cite{XZ2004} stated that the remaining cases ``appear to be rather challenging.''
In this paper (see Theorem~\ref{Berger-Coburn Schatten}), we prove that $H_f \in S_p$ if and only if $H_{\bar f}\in S_p$ when $1<p<\infty$, which we refer to as the Berger-Coburn phenomenon on $S_p$.

Bauer's proof in~\cite{Bauer} is elementary and similar to Stroethoff's characterization of compact Hankel operators in~\cite{St92}. The proof further depends heavily on the properties of the Hilbert-Schmidt norm, which makes Bauer's techniques unsuitable for the other Schatten classes. Instead, our approach is based on the characterization of the Schatten class membership of single Hankel operators, given in Theorem~\ref{S-p-criteria}, which curiously had not been established before despite analogous results for Hankel operators on the Bergman space due to Luecking~\cite{Lu92}.

While Theorems~\ref{S-p-criteria} and~\ref{Berger-Coburn Schatten} are new even in $F^2$, we prove them for Hankel operators on weighted Segal-Bargmann spaces $F^2(\varphi)$, which include all standard Segal-Bargmann and Fock-Sobolev spaces defined below. Regarding terminology, we note that $F^2(\varphi)$ are also referred to as Fock spaces and Bargmann-Fock spaces.

\subsection{Main results}
The weighted Segal-Bargmann space $F^2(\varphi)$ consists of all entire functions $f : \C^n \to \C$ that belong to $L^2(\varphi) = L^2(\C^n, e^{-2\varphi}dv)$, where $\varphi:\C^n\to \C$ is a suitable weight and $dv$ is the Lebesgue measure on $\C^n$.

In this paper, we consider real-valued weights $\varphi\in C^2(\C^n)$ that satisfy the property that there are two positive constants $m$ and $M$ such that
\begin{equation}\label{weights-a}
	m\omega_0 \le \mathrm{i} \partial \overline{\partial} \varphi \le M\omega_0
\end{equation}
in the sense of currents,  where $\omega_0 = \mathrm{i} \partial \overline{\partial}  |z|^2 $ is the Euclidean-K\"{a}hler form on $\C^n$. The expression (\ref{weights-a}) is also denoted by $\mathrm{i} \partial \overline{\partial} \varphi \simeq \omega_0$ and it simplifies to the form $m \le \Delta \varphi \le M$, where $\Delta\varphi$ is the Laplacian of $\varphi$, when $n=1$. Notice that the standard weights $\varphi(z) = \alpha|z|^2$ with $\alpha>0$ (see, e.g.,~\cite{Zh12}) satisfy~\eqref{weights-a} and each Fock-Sobolev space $F^{2,m}$  consisting of entire functions $f$ for which $\partial^{\alpha} f \in F^2 $ for all multi-indices $|\alpha|\le m$ (see~\cite{CZ2012}) can also be viewed as a $F^2(\varphi)$ with some $\varphi$ satisfying \eqref{weights-a}.

As above, we denote by $P$ the orthogonal (Bergman) projection of $L^2(\varphi)$ onto $F^2(\varphi)$. Let
$
	\Gamma = \linspan \{ K_z : z\in \C^n\},
$
where $K_z=K(\cdot, z)$ is the reproducing kernel of $F^2(\varphi)$ (see Section~\ref{reproducing kernel function}), and consider the class of symbols
$$
 	\mathcal S=\left \{f \textrm{ measurable on }\C^n: f g\in L^2(\varphi)\ \textrm{for}\ g \in \Gamma\right \}.
$$
Given $f \in \mathcal S$ and $g\in \Gamma$, the Hankel operator $H_f$ is well defined, and since $\Gamma$ is dense in $F^2(\varphi)$, it follows that $H_f = (I-P)M_f$ is densely defined on $F^2(\varphi)$. Notice that clearly $L^\infty\subset\mathcal{S}$.

To state our main results, we define
\begin{equation}\label{G-q-r}
	G_{r}(f)(z)=\inf_{h\in H(B(z,r))} \left(\frac{1}{|B(z,r)|}\int_{B(z,r)} |f-h|^2 dv \right)^{\frac{1}{2}}\quad (z\in \C^n)
\end{equation}
for $f\in L^2_{\mathrm{loc}}$ (the set of all locally square integrable functions on $\C^n$), where $H(B(z, r))$ is the set of all holomorphic functions on $B(z, r) = \{w\in\C^n : |z-w|<r\}$ and $|B(z, r)|=\int_{B(z, r)} dv$.

For $0< s\le \infty$, the space $\mathrm{IDA}^{s}$ (Integral Distance to Analytic Functions) consists of all $f\in L^2_{\mathrm {loc}}$ such that
\begin{equation}\label{IDA definition}
	\|f\|_{\mathrm{IDA}^{s}} =  \|  G_{r}(f)\|_{L^s}<\infty
\end{equation}
for some $r>0$. We write $\BDA$ for $\IDA^\infty$. The space $\mathrm{VDA}$  consists of all $f\in L^2_{\mathrm {loc}}$ such that
$
	\lim_{z\to \infty} G_{r}(f)(z)=0
$
for some $r>0$.

The spaces $\IDA^s$ with $s<\infty$ (and their generalizations $\IDA^{s, p}$ with the convention that $\IDA^s = \IDA^{s,2}$) were introduced in~\cite{HV20}, while the notion of bounded distance to analytic functions (BDA) was introduced by Luecking~\cite{Lu92} in the context of the Bergman space.

\medskip

We can now state our main results on Schatten class Hankel operators.

\begin{theorem} \label{S-p-criteria}
Let $0<p<\infty$ and suppose that $\varphi\in C^2(\Cn)$ is real valued with $\mathrm{i} \partial \overline{\partial} \varphi \simeq \omega_0$. Then for  $f\in \mathcal S$, the following statements are equivalent:
\begin{itemize}
\item[(A)] $H_f : F^2(\varphi) \to L^2(\varphi)$ is in $S_p$.

\item[(B)] $f\in \mathrm{IDA}^{p}$.

\item[(C)] $\int_{\Cn} \left\|H_f(k_z)\right\|^p dv(z)<\infty$, where $k_{ z}(\cdot)=\frac{K(\cdot,z)}{\sqrt{K(z,z)}}$ is the normalized reproducing kernel.
\end{itemize}
Furthermore,
\begin{equation}\label{bounded-g}
	\|H_f\|_{S_p}\simeq \left\|   f \right\|_{\mathrm{IDA}^{p}}\simeq \left\{\int_{\Cn} \left\|H_f(k_z)\right\|^p dv(z)\right\}^{\frac1p} .
\end{equation}
\end{theorem}

\begin{theorem} \label{Berger-Coburn Schatten}
Suppose $\varphi\in C^2(\C^n)$ is  real valued,  $\mathrm{i} \partial \overline{\partial} \varphi \simeq \omega_0$, and $1< p<\infty$.   Then for $f\in L^\infty$,  $H_f\in {S_p}$ if and only if $H_{\bar f} \in {S_p}$ with the $S_p$-norm estimate
\begin{equation}\label{S-p-norm}
	\|H_{\bar f}\|_{S_p}\le C  \|H_{ f }\|_{S_p},
\end{equation}
where the constant $C$ is independent of $f$.
\end{theorem}

Notice that Theorem~\ref{Berger-Coburn Schatten} fails in general if the symbol $f$ is not bounded---see~\cite{Bauer} for an example.

\subsection{Outline and further results}
In the next section we provide preliminaries on the reproducing kernel, which includes global and local estimates, a consequence of H\"ormander's existence theorem, and we also extend the decomposition theorem in~\cite{HV20} for $\IDA$ functions. In Section~\ref{Schatten class Toeplitz operators}, we briefly introduce Toeplitz operators and state a description of their Schatten class properties. Section~\ref{boundedness and compactness} extends our recent results in~\cite{HV20} on boundedness and compactness of Hankel operators, and we compare them to the results of Stroethoff~\cite{St92} for {\it bounded symbols} in the setting of the classical Segal-Bargmann space $F^2$.

In Section~\ref{proof of S-p-criteria}, we prove our characterization of Schatten class Hankel operators using the decomposition theorem and other preliminary results, theory of Schatten class Toeplitz operators, and various estimates together with the general theory of Schatten class operators. As a consequence, when $\varphi(z) = \alpha|z|^2$,
we obtain a characterization in a familiar form that agrees with one of the main results in~\cite{Bauer} obtained previously only when $p=2$.

In Section~\ref{proof of Simul-S-p}, we recall briefly the results in~\cite{I2011, XZ2004} on the simultaneous membership in $S_p$ of the Hankel operators $H_f$ and $H_{\overline{f}}$ acting on $F^2$, and extend them to $F^2(\varphi)$ using Theorem~\ref{S-p-criteria}.

In Section~\ref{Berger-Coburn phenomenon}, we prove our result on the Berger-Coburn phenomenon using the Ahlfors-Beurling operator (to obtain the estimates $\| \partial f\|_{L^p} \lesssim \| \overline{\partial} f\|_{L^p}$) together with our characterization of Schatten class Hankel operators.

On the methodological level, it is worth noting that the previous techniques employed in~\cite{Bauer, BC1987, HaVi20, XZ2004}, such as the use of Weyl operators, limit operators or explicit formulas for the Bergman kernel, seem insufficient to obtain our characterizations even in setting of the classical Segal-Bargmann space $F^2$. In relation to the Bergman space $A^2$, Luecking~\cite{Lu92} gave a characterization for $H_f$ on $A^2$ of the unit disk to be in $S_p$ when $1\le p<\infty$, and further indicated that his proofs can be extended to handle any {\it bounded} strongly pseudoconvex domain in $\C^n$. Luecking's work and in particular the concept that he referred to as the ``bounded distance to analytic functions'' are of fundamental importance to our characterizations (see also~\cite{LL95} for bounded strongly pseudoconvex domains). 

While several aspects of the theory of Hankel operators on the Segal-Bargmann space are different from $A^2$ and $H^2$, there are also many similarities, such as the role of $\BMO$-type spaces and their decompositions, which we use to closely track the results for the Bergman space.

\section{Preliminaries}\label{preliminaries}

\subsection{The reproducing kernel function}\label{reproducing kernel function}

Let $\varphi \in C^2(\C^n)$ be a real-valued weight such that $i\partial\overline{\partial} \varphi \simeq \omega_0$, see~\eqref{weights-a}. Most of the basic properties of $F^2(\varphi)$ can be derived from the following weighted Bergman inequality (see Proposition~2.3 of~\cite{SV12} for further details).

\begin{lemma}\label{submean-value}
For each $r>0$, there is a constant $C>0$ such that
$$
 	\left| f(z) e^{-\varphi(z)}\right|^2 \le C \int_{B(z, r)} \left| f(\xi) e^{-\varphi(\xi)}\right|^2 dv(\xi)
$$
for all $f\in F^2(\varphi)$.
\end{lemma}

It follows from the preceding lemma that for any $z\in\C^n$, the mapping $f\mapsto f(z)$ is a bounded linear functional on $F^2(\varphi)$ and hence there is a unique $K_z$ in $F^2(\varphi)$ which satisfies the reproducing property $f(z) = \langle f, K_z\rangle$ for all $f\in F^2(\varphi)$. The function $K_z$ is referred to as the reproducing kernel of $F^2(\varphi)$. It is often called the Bergman kernel.

Lemma~\ref{submean-value} also implies that  $F^2(\varphi)$ is a closed subspace of $L^2(\varphi)$. We denote by $P$ the orthogonal projection of $L^2(\varphi)$ onto $F^2(\varphi)$. Notice that
$$
	Pf(z) = \langle Pf, K_z\rangle = \int_{\C^n} f(w) K(z,w) e^{-2\varphi(w)}dv(w)
$$
for $f\in L^2(\varphi)$ and $z\in\C^n$.

If $\varphi(z) = \alpha|z|^2$ is a standard weight with $\alpha>0$, then it is easy to see that
$$
	\left|K(z,w)\right| e^{-\alpha|z|^2 - \alpha|w|^2} = e^{-\alpha|z-w|^2}
$$
for $z,w\in \C^n$. For the general weights $\varphi$ that we consider, this quadratic decay is known not to hold (even in dimension one), and it is, in fact, expected to be very rare (see~\cite{Ch91}). However, it turns out that the following estimates for the reproducing kernel will be sufficient for us.

\begin{lemma}\label{basic-est}
There exist positive constants $\theta$ and $C_1$, depending only on $n$, $m$ and $M$ such that
\begin{equation}\label{basic-est-a}
	\left|K(z,w)\right|\leq  C_1 e^{ \varphi(z)+\varphi(w)} e^{-\theta|z-w|}\,  \textrm{ for all }\, z, w\in{{\C}}^n,
\end{equation}
and there exists positive constants $C_2$ and $r_0 $ such that
\begin{equation}\label{basic-est-b}
	\left|K(z,w)\right|\geq C_2 e^{ \varphi(z)+\varphi(w)}
\end{equation}
for $z\in{{\C}}^n$ and  $w\in B\left(z, r_0\right)$.
\end{lemma}

The estimate \eqref{basic-est-a}~appeared in \cite{Ch91} for $n=1$ and in~\cite{De98} for $n\ge 2$, while the inequality \eqref{basic-est-b} can be found in~\cite{SV12}. Notice that Lemma~\ref{basic-est} implies that
\begin{equation}\label{K-estimate}
	K(z, z)\simeq e^{2\varphi(z)}, \  z\in \Cn.
\end{equation}
For  $z\in {{\C}}^n$, recall that we write
\begin{equation*}
	k_{ z}(\cdot)=\frac{K(\cdot,z)}{\sqrt{K(z,z)}}
\end{equation*}
for the normalized reproducing kernel.  Then
\begin{equation}\label{k-z-est}
	\left |k_z(\xi) e^{-\varphi(\xi)}\right| \le C e^{- \theta |z-\xi|}, \ \  \lim_{|z|\to \infty} k_z(\xi)=0
\end{equation}
uniformly in $\xi$ on compact subsets of $\Cn$, and
\begin{equation}\label{nor-bergman}
	\frac 1 C e^{ \varphi(z)} \le  \|K_z\|_{p, \varphi}  \le C  e^{ \varphi(z)}, \ \  \frac1C\le\|k_{z}\|_{p,\varphi}\le C
\end{equation}
for $z\in \C^n$. Here $\|\cdot\|_{p,\varphi}$ stands for the norm of $L^p(\varphi) = L^p(\C^n, e^{-p\varphi}dv)$ when $1\le p<\infty$, $\|f\|_{\infty, \varphi} = \|fe^{-\varphi}\|_{L^\infty}$, and we write $\|\cdot\|_{2,\varphi} = \|\cdot\|$ for simplicity throughout.

We record one more estimate that will be needed for our study of Hankel operators. For this purpose, denote by $L^2_{(0, 1)}(\varphi)$ the family of all $(0, 1)$-forms on $\Cn$ with coefficients in $L^2(\varphi)$.

\begin{lemma}[\textbf{H\"ormander}]\label{Hormander}
Suppose that $\varphi\in C^2(\Cn)$ is  real valued and  $\mathrm{i} \partial \overline{\partial} \varphi \simeq \omega_0$. Then there is a constant $C>0$ such that for every $\overline \partial$-closed $(0, 1)$-form $\omega\in L^2_{(0, 1)}(\varphi)$, there exists a solution $u$ of $\overline \partial u =\omega$ for which
$$
	\int_{\Cn} \left |ue^{-\varphi}\right |^2 dv\le C \int_{\Cn} \left |\omega e^{-\varphi}\right |^2 dv.
$$
\end{lemma}
\begin{proof}
Let $\Omega =\Cn$. The assumption $\mathrm{i} \partial \overline \partial \varphi (z)\ge m \,\mathrm{i} \partial \overline \partial |z|^2$ implies that $2m$ is a lower bound for the plurisubharmonicity of $2\varphi$. Now Theorem 2.2.1 of~\cite{Ho65} completes the proof.
\end{proof}

\subsection{Lattices and separated sets in ${\Cn}$}
 A sequence $\{w_j\}$ of distinct points in $\Cn$ is called separated if
$$
	\delta(\{w_j\})=\inf_{j\neq k}|w_j-w_k|>0.
$$
For $r>0$, we call a sequence $\{w_j\}$ in   $\Cn$ an $r$-lattice if $\bigcup_{j} B(w_j, r)=\Cn$ and $B(w_j, \frac{r}{2\sqrt{n}})\bigcap B(w_k, \frac r {2\sqrt{n}})= \emptyset$ whenever $j\neq k$. 

Given $r>0$ and $w_1\in \Cn$, set
\begin{equation}\label{lattice-c}
	\Lambda= \left \{w_1+ \frac r {\sqrt{n}} ( m+ \textrm i s): m, s  \in {\mathbb Z}^n \right \}.
\end{equation}
It is easy to see that $\Lambda$ is an $r$-lattice. For $K\in \mathbb N$ fixed, we write
$$
	\left\{ w_1+ (z_1, \ldots, z_n)\in \Lambda:  0
	\le  \mathrm {Re}\, z_j, \mathrm {Im}\, z_j <K \frac r {\sqrt{n}}\right\}
	= \left\{w_1, \ldots, w_{K^{2n}}\right\},
$$
and for $1\le k\le K^{2n}$,
$$
	\Lambda_k=  \left \{w_k+ K\frac r {\sqrt{n}} ( m+ \textrm i s): m, s \in {\mathbb Z}^n \right \}.
$$
Then
\begin{equation}\label{Omega-docomposition}
	\Lambda= \bigcup_{k=1}^{K^{2n}} \Lambda_k, \quad \Lambda_j\cap \Lambda_k
	= \emptyset \ \textrm{ if } \ j\neq k,\quad |a-b|\ge \frac {Kr}{\sqrt{n}}  \quad \textrm {if}\ a, b \in \Lambda_k\ {\rm and}\ a\ne b.
\end{equation}

For $f, e\in L^2(\varphi)$, the tensor product $f\otimes e$ as a rank one operator on $L^2(\varphi)$ is defined by
$$
	f\otimes e(g) = \langle g, e \rangle f, \ \ g\in L^2(\varphi).
$$

\begin{lemma}\label{tensor-product}
Given $r>0$, there is some constant $C>0$ such that if $ \Lambda $ is  a separated    set in $\Cn$ with $\delta (\Lambda)\ge r$     and if $\{e_a: a\in \Lambda\}$ is an orthonormal set in $L^2(\varphi)$, then
$$
	\left\|\sum_{a\in \Lambda} k_{a}\otimes e_a\right\|_{L^2(\varphi)\to L^2(\varphi)} \le C.
$$
\end{lemma}
\begin{proof}
With the same proof as that of  Lemma 2.4 in \cite{HL14}, we get
$$
	\left\| \sum_{a\in \Lambda} \lambda_a k_a\right \|\le C \left \|\{\lambda_a\}_{a\in \Lambda}\right \|_{l^2},
$$
where the constant depends only on the separation constant $\delta(\Lambda)$. In addition, for $g\in L^2(\varphi)$, Parseval's identity implies
$$
	\sum_{a\in \Lambda} \left| \langle g, e_a\rangle\right|^2 \le \|g\|^2<\infty .
$$
Therefore, we have
$$
	\left\| \left( \sum_{a\in \Lambda} k_{a}\otimes e_a \right) (g)\right\|^2
	=\left\| \sum_{a\in \Lambda} \langle g, e_a\rangle k_{a} \right\|^2
	\le C \sum_{a\in \Lambda} \left| \langle g, e_a\rangle\right|^2\le  C  \|g\|^2,
$$
which completes the proof.
\end{proof}

\subsection{Properties of $\IDA$}
The spaces $\IDA^{s}$ were defined above in~\eqref{IDA definition} and here we list their basic properties. We start with a remark that follows from Corollary~3.8 of~\cite{HV20} when $s\ge 1$ while the other cases can be proved similarly.

\begin{remark}\label{independence}
Let $0<s<\infty$. Then the spaces $\IDA^{s}$, $\BDA$ and $\VDA$ are independent of $r$ and different values of $r$ give equivalent norms on each space.
\end{remark}

For $f\in L^2_{\textrm{loc}}$, set
$$
	M_{2, r}(f)(z)= \left\{ \frac 1 {|B(z, r)|} \int_{B(z, r)} |f|^2 dv\right\}^{\frac 12}.
$$
Notice that, for $r>0$ fixed, by Lemma 3.5 in \cite{HV20}, there is a constant $C$ such that for  $f\in L^2_{\mathrm {loc}}$ it holds that $f=f_1+f_2$ with $f_1\in C (\C^n)$ and
\begin{equation}\label{decomp-a}
	|\overline \partial f_1(z)|+ M_{2, r}(|\overline \partial f_1|)(z)  + M_{2, r}(f_2)(z) \le C G_{2r}(f)(z), \ z\in \C^n.
\end{equation}
In addition, the following decomposition theorem plays an important role in our analysis.
\begin{theorem}\label{IDA-space}
Suppose $0< s\le \infty$, and $f\in L^2_{\mathrm {loc}}$. Then $f\in  \mathrm{IDA}^{s}$ if and only if $f$ admits a decomposition $f=f_1+f_2$ such that
\begin{equation}\label{q-less-then-p-m}
	f_1\in C^2(\Cn), \, \, |\overline \partial f_1|+  M_{2,r}(\overline \partial f_1)  + M_{2, r}(f_2) \in L^{s}
\end{equation}
for some (or any) $r>0$. Furthermore,
\begin{equation}\label{norm-euival}
	\|f\|_{\mathrm{IDA}^{s}} \simeq  \inf \left \{ \| \overline \partial f_1 \|_{L^{s}}   + \|M_{2, r}(f_2)\|_{L^{s}}\right\}
\end{equation}
where the infimum is taken over all possible decompositions $f=f_1+f_2$ that satisfy  \eqref{q-less-then-p-m} with a fixed $r$.
\end{theorem}

\begin{proof}
When $1\le s<\infty$, the conclusion is only a special case $q=2$ of~Theorem 3.6 and Lemma 3.5 in \cite{HV20}. A careful check of their proofs shows that the remaining cases $0<s<1$  and $s=\infty$  can be proved similarly.
\end{proof}

\subsection{Schatten classes}\label{Schatten classes} We finish this section with a brief look at Schatten classes. Recall that for a bounded linear operator $T : H_1 \to H_2$ between two Hilbert spaces, the singular values $s_j(T)$ of $T$ are defined by
\begin{equation}\label{e:singular values}
	s_j(T) = \inf \{ \|T-K\| : K : H_1\to H_2,\ \rank K\le j\},
\end{equation}
where $\rank K$ denotes the rank of $K$. The operator $T$ is compact if and only if $s_j(T) \to 0$. For $0<p<\infty$, we say that $T$ is in the Schatten class $S_p$ and write $T\in S_p(H_1, H_2)$ if $\|T\|_{S_p}^p = \sum_{j=0}^\infty \left(s_j(T)\right)^p < \infty$, which defines a norm when $1\le p<\infty$ and a quasinorm otherwise. Note that $S_p$ are also called the Schatten-von Neumann classes or trace ideals (see~\cite{Zh07} for further details). 

\section{Schatten class Toeplitz operators}\label{Schatten class Toeplitz operators}

Given a   Borel measure  $\mu$ on ${\mathbb C}^n$, we define the Toeplitz operator $T_\mu$ with  symbol $\mu$ as
$$
	T_{\mu}f(z)=\int_{{\mathbb C}^n}K(z,w)f(w)e^{-2\varphi(w)}d\mu(w), \verb#   # f\in F^2(\varphi) \textrm{ and } z\in{\mathbb C}^n.
$$
When $d\mu(z)=g(z)\,dv(z)$ and $g$ is a complex-valued function, the induced Toeplitz operator is denoted by $T_g$.

\medskip

Given an operator $T\in \mathcal B(F^2(\varphi), L^2(\varphi))$, we set
$ \widetilde{T}(z) =\langle Tk_z, k_z\rangle$. For a positive Borel measure $\mu$  on $\Cn$ and $r>0$, we define
$$
 	\widetilde{\mu}(z)= \int_{\Cn} |k_z|^2 e^{-2\varphi} d\mu
$$
and since $|B(z,r)| \simeq r^{2n}$, we simply set
$$
	\widehat \mu_r(z) = \int_{B(z, r)} d\mu \quad(z\in \Cn).
$$
For a positive Toeplitz operator $T_{\mu}\in \mathcal B(F^2(\varphi), L^2(\varphi))$, it is easy to  verify that $\widetilde{T_\mu} = \widetilde{\mu}$.

\medskip

The proof of the following lemma can be found in \cite{HL14} and \cite{IVW15}.

 \begin{lemma}\label{carleson-meas}
 Let $\mu$ be a  positive Borel measure on $\Cn$ and let $0<p\le \infty$. Then the following statements are equivalent:
\begin{itemize}
\item[(A)] $\widetilde{\mu} \in L^p$.

\item[(B)] $\widehat \mu_r \in L^p$ for some (or any) $r>0$.

\item[(C)] $\{\widehat \mu_r(a_j)\}_{j=1}^\infty \in l^p$ for some (or any) $r$-lattice $\{a_j\}_{j=1}^\infty$.
\end{itemize}
Furthermore, it holds that
$$
	\|  \widetilde{\mu} \|_{ L^p} \simeq \|  \widehat \mu_r\|_{ L^p}
	\simeq \left\| \{\widehat \mu_r(a_j)\}_{j=1}^\infty\right\|_{l^p}.
$$
\end{lemma}

In our analysis we need the following result on Schatten class Toeplitz operators. It was proved in~\cite{IVW15} for the generalized weights and in~\cite{Zh12} for the standard weights.

\begin{theorem}\label{Sp-Toep}
Let $0<p<\infty$, $\mu$ be a positive Borel measure on $\Cn$ and suppose that $\varphi\in C^2(\Cn)$ is real valued with $\mathrm{i} \partial \overline{\partial} \varphi \simeq \omega_0$. Then the Toeplitz operator $T_\mu$ on $F^2(\varphi)$ belongs to $S_p$ if and only if $  \widehat \mu_r\in L^p$ for some (or any) $r>0$. Furthermore,
\begin{equation}\label{Sp-norm}
       \|T_\mu\|_{S_p} \simeq \|\widehat \mu_r\|_{L^p}.
\end{equation}
\end{theorem}

\section{Boundedness and compactness of Hankel operators}\label{boundedness and compactness}

While the thrust of our present work is in the Schatten class properties, we also extend some of the recent results in~\cite{HV20} on boundedness and compactness of Hankel operators on $F^2(\varphi)$.

The collection of all bounded  (and compact)  operators from $H_1$ to $H_2$ is denoted by ${\mathcal B}(H_1,H_2)$ (and by ${\mathcal K}(H_1,H_2)$, respectively). The corresponding  operator norm is denoted by $\|T\|_{H_1\to H_2}$.

\begin{lemma}\label{integral-est}  Suppose $0<p\le1$, $0<s<\infty$ and $r>0$. There is a constant $C$ such that, for  $\mu$  a positive Borel measure on $\Cn$, $\Omega$  a domain in $\Cn$, and $g\in H(\Cn)$, it holds that
$$
	\left(\int_{\Omega} \left | g e^{-\varphi}\right| ^s  d\mu\right)^p
	\le  C \int_{\Omega^+_r} \left | g e^{-\varphi}\right|^{sp} \widehat{\mu}_r^p dv,
$$
where $\Omega^+_r = \bigcup_{\{z\in \Omega\}} B(z, r)$.
\end{lemma}

When $\Omega=\Cn$ and  $p=1$, the preceding result is just Lemma 2.2 in \cite{HL14}. For general $\Omega$ and $0<p\le 1$, the proof is similar to that of Lemma~4.2 in~\cite{HV20}, although we note that the weights in the present work are slightly more general.

\begin{lemma}\label{IDA-Hankel}
Suppose $0<p\le \infty$, and $f\in \mathcal S \cap \mathrm{IDA}^{p}$ with the decomposition $f=f_1+f_2$ as in \eqref{decomp-a}.
Then $H_{f_1}$ and $H_{f_2}$  are well defined on $\Gamma = \linspan \{ K_z : z\in \C^n\}$, and
\begin{equation}\label{kernel-image-d}
	\|H_{f_1}(g)\| \le C \|g \overline \partial f_1\|  \ \textrm{ and } \  \|H_{f_2}(g)\| \le C \|g   f_2\| \ \ \textrm{ for }\ g \in \Gamma.
\end{equation}
\end{lemma}

\begin{proof}
For $g\in \Gamma$ and $z\in \Cn$, taking $p = s=1$ and replacing $\varphi$ with $2\varphi$ in  Lemma \ref{integral-est}, we get
$$
	\int_{\Cn} |g  K_z|e^{-2\varphi} |f_2| dv
	\le C  \int_{\Cn} |g  K_z|e^{-2\varphi} M_{1, r}(f_2) dv.
$$
Notice that $ M_{t, r}(f_2)$ is increasing with $t$.  If $p\ge 1$, with $p'$ being the conjugate of $p$, we have
$$
	\int_{\Cn} |g f_2| |K_z|e^{-2\varphi}dv
	\le C \| M_{2, r}(f_2)\|_{L^p} \|g\|_{2p', \varphi} \|K_z\|_{2p', \varphi}<\infty.
$$
If $0<p<1$, by Lemma \ref{integral-est}  again,
\begin{align*}
	\left( \int_{\Cn} |g  K_z|e^{-2\varphi} |f_2| dv \right)^p
	&\le C  \int_{\Cn} |g  K_z|^p e^{-2p\varphi} M_{1, r}^p(f_2) dv\\
	&\le C \| M_{2, r}(f_2)\|_{L^p}^p \|g\|_{\infty, \varphi}^p \|K_z\|_{\infty, \varphi}^p<\infty.
\end{align*}
This implies that $H_{f_2}$,  and hence also $H_{f_1}= H_{f}-H_{f_2}$, are both well defined on $\Gamma$.

Now for $g\in \Gamma$, if $0<p\le 2$, then
\begin{equation}\label{norm-f1g-a}
\begin{split}
	\left \|g {\overline \partial} f_1\right \|^p
	&\le  C  \int_{\Cn} \left | g  e^{-\varphi} \right |^{  p  } M_{2, r}  (|{\overline \partial} f_1|)^{
     p } dv \\
	&\le  C \| M_{2, r}  (|{\overline \partial} f_1|)\|_{L^p}^p  \|g \|_{\infty, \varphi}^p <\infty.
 \end{split}
\end{equation}
 If $p>2$,  applying H\"older's inequality with $t=\frac p 2$ and $t'$ we obtain
\begin{equation}\label{norm-f1g-b}
\begin{split}
	\left \|g {\overline \partial} f_1\right \|
	&\le  C \left\{ \int_{\Cn} \left | g  e^{-\varphi} \right |^{ 2  } M_{2, r}  (|{\overline \partial} f_1|)^{
    2 } dv \right\}^{\frac 12} \\
	&\le C \| M_{2, r}  (|{\overline \partial} f_1|)\|_{L^p}  \, \|g\|_{2t', \varphi} <\infty.
\end{split}
\end{equation}
Hence $g{\overline \partial} {f_1}$ is a $\overline \partial$-closed $(0, 1)$-form with $L^2(\varphi)$ coefficients.   Notice also that since $\overline \partial H_{f_1} (g)= g\overline\partial {f_1}$ and  $ H_{f_1} (g) = f_1g -P(f_1g)  \bot F^2(\varphi)$,   $H_{f_1}(g)$ is the canonical solution of the equation $\overline \partial  u = g\overline \partial f_1$.
Now, by Lemma \ref{Hormander} there is a solution $u_0\in L^2(\varphi)$ of $\overline \partial u= g\overline \partial f_1$ such that $\|u_0\| \le C \|g\overline\partial f_1\|$. Hence, $u_0- H_{f_1} g \in F^2(\varphi)$, which implies that $H_{f_1}g \perp  u_0 - H_{f_1}g$. Therefore,
\begin{equation}\label{f_1-estimate}
	 \|  H_{f_1}g\|^2 =\|u_0\|^2 -\|u_0- H_{f_1} g\|^2 \le \|u_0\|^2
	\le C\left  \|g {\overline \partial} f_1\right \|^2.
\end{equation}
Similarly to \eqref{norm-f1g-a} and \eqref{norm-f1g-b}, we can show that $\left  \|g   f_2\right \|<\infty$, and so
$$
	\|  H_{f_2}(g)\| \le  \left  \|g   f_2\right \|<\infty,
$$
which completes the proof.
\end{proof}

Before stating our result on boundedness and compactness, we note that previously Li~\cite{Li94} used the Henkin-Ramirez formula to obtain conditions for boundedness of Hankel operators acting on the classical Segal-Bargmann space. However,  some basic estimates involving the Henkin-Ramirez formula in Li's approach are unavailable for the more general weights $\varphi$ that we consider.

\begin{theorem} \label{bounded}
 Suppose that  $\varphi\in C^2(\Cn)$ is  real valued and $\mathrm{i} \partial \overline{\partial} \varphi \simeq \omega_0$.   Then for $f\in \mathcal S$,  $H_f\in {\mathcal B}(F^2(\varphi), L^2(\varphi))$ if and only if  $f\in \mathrm{BDA}$;  and $H_f\in {\mathcal K}(F^2(\varphi),  L^2(\varphi))$ if and only if $f\in \mathrm{VDA}$. Furthermore, for $r>0$,
\begin{equation*}
	\|H_f\|_{F^2(\varphi)\to L^2(\varphi)}\simeq \left\|   f \right\|_{\mathrm{BDA}}.
\end{equation*}
\end{theorem}
\begin{proof}
Suppose $f\in \BDA$. As in \eqref{decomp-a}, we decompose $f=f_1+f_2$  with
$$
	f_1\in C^2(\Cn), \quad |\overline \partial f_1|+ M_{2, r}(f_2) \le C G_{2r}(f)(z).
$$
For $g\in \Gamma$, by Lemma \ref{IDA-Hankel},
\begin{equation}\label{f--1-estimate}
	\|  H_{f_1}(g)\|
	\le C\left  \|g {\overline \partial} f_1\right \|
	\le C \|{\overline \partial} f_1\|_{L^\infty} \|g \|.
\end{equation}
Therefore,  $H_{f_1}\in {\mathcal B}(F^2(\varphi), L^2(\varphi))$ with the norm estimate
$$
	\|H_{f_1}\|_{F^2(\varphi)\to L^2(\varphi)} \le C \|{\overline \partial} f_1\|_{L^\infty}.
$$
Similarly,
\begin{equation}\label{f--2-estimate}
	\|  H_{f_2}(g)\|
	\le  \left  \|g   f_2\right \|  \le C \| M_{2, r}(f_2) \|_{L^\infty} \|g \|.
\end{equation}
Therefore,   $H_f\in {\mathcal B}(F^2(\varphi),  L^2(\varphi))$ with the norm estimate
$$
	\|H_f\|_{F^2(\varphi)\to L^2(\varphi)}\le C \left\|   f \right\|_{\mathrm{BDA}}.
$$

Conversely, suppose $H_f\in {\mathcal B}(F^2(\varphi),  L^2(\varphi))$. Then for $0<r\le r_0$, by Lemma \ref{basic-est},
\begin{equation}\label{key-estimate}
	G_{r}(f)(z) \le C \left\{ \int_{B(z, r)} \left|f- \frac 1{k_z}P(f k_z) \right|^2 dv  \right\}^{\frac 12}
	\le C \|H_f(k_z)\|.
\end{equation}
This and the fact that $\|k_z \| =1 $ implies  $f\in \BDA$ with
$$
	\left\|   f \right\|_{\mathrm{BDA}} \le C \|H_f\|_{F^2(\varphi)\to L^2(\varphi)}.
$$

The proof for the compactness can be carried out as that of Theorem 4.1 in \cite{HV20}.
\end{proof}

Our next theorem is an analog of Stroethoff's main result in~\cite{St92} for the generalized weights and unbounded symbols.

\begin{theorem}\label{Stroethoff}
Suppose $\varphi\in C^2(\Cn)$ is  real valued and $\mathrm{i} \partial \overline{\partial} \varphi \simeq \omega_0$.   Then for $f\in \mathcal S$,  $H_f\in {\mathcal B}(F^2(\varphi), L^2(\varphi))$ if and only if
$$
	\sup_{z\in \Cn} \left\| \left(I-P \right)(fk_z) \right\|<\infty
$$
with the norm estimate
\begin{equation}\label{bounded-m}
	\|H_f\|_{F^2(\varphi)\to L^2(\varphi)}\simeq \sup_{z\in \Cn} \left\| \left(I-P \right)(fk_z) \right\|.
\end{equation}
Further, $H_f\in {\mathcal K}(F^2(\varphi),  L^2(\varphi))$ if and only if
$$
	\lim_{z\to \infty} \left\| \left(I-P \right)(fk_z) \right\| =0.
$$
\end{theorem}

\begin{proof} Suppose $f\in \mathcal S$. If $H_f\in {\mathcal B}(F^2(\varphi), L^2(\varphi))$, then trivially
$$
	\left\| \left(I-P \right)(fk_z) \right\| \le \|H_f\|_{F^2(\varphi)\to
	L^2(\varphi)} \|k_z\|\le C  \|H_f\|_{F^2(\varphi)\to L^2(\varphi)}.
$$
Conversely, by Theorem \ref{bounded} and the estimate  (\ref{key-estimate}), we have
$$
	\|H_f\|_{F^2(\varphi)\to L^2(\varphi)}
	\simeq  \| G_{r}(f)\|_{L^\infty}
	\le   \sup_{z\in \Cn} \left\| \left(I-P \right)(fk_z) \right\|.
$$
The other equivalence can be proved similarly.
\end{proof}

For a fixed $z\in \Cn$, define the shift  $\tau_z:\Cn \to \Cn$ by $\tau_z(w)= w+z$.  Then $f\mapsto f\circ \tau_z(w)= f(w+z)$ for $w\in\Cn$. In the classical Segal-Bargmann space $F^2$, we have
\begin{equation}\label{shift-a}
	\left\| H_f (k_z)  \right\| = \left\|  (I-P)\left( f\circ \tau_z \right)  \right\|
\end{equation}
for all $f\in \mathcal S$ (see Corollary 1.1 in \cite{Bauer}). Using this observation, we obtain the following corollary.

\begin{corollary}\label{Stroethoff-a}
Let $\varphi(z)=\frac \alpha 2 |z|^2$. For $f\in \mathcal S$,  $H_f\in {\mathcal B}(F^2(\varphi) , L^2(\varphi) )$ if and only if
$
  \sup_{z\in \Cn} \left\| \left(I-P \right)(f\circ \tau_z) \right\|<\infty
$ with the norm estimate
\begin{equation}\label{bounded-n}
	\|H_f\|_{F^2(\varphi)\to
	L^2(\varphi)}\simeq \sup_{z\in \Cn} \left\| \left(I-P \right)(f\circ \tau_z) \right\|.
\end{equation}
Further, $H_f\in {\mathcal K}(F^2(\varphi),  L^2(\varphi))$ if and only if
$$
	\lim_{z\to \infty} \left\| \left(I-P \right)(f\circ \tau_z) \right\|=0.
$$
\end{corollary}

\begin{remark}
(i) Notice that the preceding result on compactness implies the analogous results in~\cite{HaVi20} and \cite{St92} where the symbols were assumed to be bounded.

(ii) It should be noted that the norm estimate in~\eqref{bounded-n} is similar to the estimate in Luecking's main result for the Bergman space (and also for the  standard weights); see Thereom~1(b) of~\cite{Lu92}.

(iii) It would be interesting to know whether the previous corollary remains true for more general weights.
\end{remark}

\section{Proof of Theorem~\ref{S-p-criteria}}\label{proof of S-p-criteria}
In this section we prove our characterization of Schatten class Hankel operators. For this purpose we need one more lemma.
Given $a\in \Cn$ and  $r>0$, let $L^2(B(a, r), e^{-2\varphi} dv)$  be the Lebesgue space on $B(a, r)$ with respect to measure $e^{-2\varphi} dv$, and let  $A^2(B(a, r), e^{-2\varphi} dv)$ be the weighted Bergman space of all holomorphic functions in the space $L^2(B(a, r), e^{-2\varphi} dv)$. We denote by $P_{a, r}$ the orthogonal projection of $L^2(B(a, r), e^{-2\varphi} dv)$ onto $A^2(B(a, r), e^{-2\varphi} dv)$.

Given $f\in L^2(B(a, r), e^{-2\varphi} dv)$, we extend $P_{a, r}(f)$ to $\Cn$ by setting
$$
	P_{a, r}(f)|_{\Cn\backslash B(a, r)} =0.
$$
It is easy to verify that
$$
	P_{a, r}^2 f= P_{a, r}f \quad \textrm{ and } \quad
	\langle f,  P_{a, r}g\rangle  = \langle P_{a, r} f, g\rangle
$$
for $f, g\in L^2(\varphi)$.

\begin{lemma}\label{local-projection}
For $f, g\in L^2(\varphi)$, it holds that
\begin{equation}\label{local-proj-a}
	\langle f-Pf, \, \chi_{B(a , r)} g-P_{a , r} g \rangle
	= \langle f- P_{a , r} f, \,  \chi_{B(a , r)} g-P_{a , r} g \rangle.
\end{equation}
\end{lemma}

\begin{proof}
For $ h \in F^2(\varphi)$, it is trivial that
$
   P_{a , r} (h)= \chi_{B(a , r)} h
$.
Then for $f, g\in L^2(\varphi)$, we have $\langle h, \chi_{B(a , r)} g-P_{a , r} g\rangle =0$, and hence
\begin{align*}
	\langle f-Pf, \,\chi_{B(a , r)} g-P_{a , r} g \rangle
	&=\langle \chi_{B(a , r)} f, \, \chi_{B(a , r)} g-P_{a , r} g \rangle\\
   	&=\langle \chi_{B(a , r)} f- P_{a , r} f, \, \chi_{B(a , r)} g-P_{a , r} g \rangle.
\end{align*}
From this (\ref{local-proj-a}) follows.
\end{proof}

\begin{proof}[Proof of Theorem~\ref{S-p-criteria}]
\textbf{ (B)$\Rightarrow$(A).} For $f\in \IDA^{p}$, decompose $f$ as in \eqref{decomp-a}.
Notice that by Lemma \ref{IDA-Hankel}, both  $H_{f_1}$ and $H_{f_2}$ are well defined on $F^2(\varphi)$.
We claim that
\begin{equation}\label{H-f1-e}
	  \left\| H_{f_1}\right\|_{S_p} +  \left\| H_{f_2}\right\|_{S_p}  \le C \|f\|_{\mathrm{IDA}^{p}},
\end{equation}
 which gives the desired estimate
\begin{equation} \label{Schatten-p-a}
	\left\| H_{f}\right\|_{S_p} \le  C\left( \left\| H_{f_1}\right\|_{S_p} +  \left\| H_{f_2}\right\|_{S_p}\right)\le C \|f\|_{\mathrm{IDA}^{p}}.
\end{equation}
To prove (\ref{H-f1-e}), set $\phi= |\overline \partial f_1|$ or $\phi= |f_2|$. By Theorem~\ref{Sp-Toep}, the positive Toeplitz operator $T_{|\phi|^2}$ is in $S_{\frac p2}$. Now we consider the multiplication operator $M_\phi: F^2(\varphi) \to L^2(\varphi)$ defined as
$$
	M_\phi(f)= \phi f.
$$
For $\phi= |\overline \partial f_1|$, using the inequality in \eqref{norm-f1g-a} with the exponential $\frac{2p}{2+p} <2 $ and H\"older's inequality, for $g\in F^2(\varphi)$ we have
\begin{align*}
	\left\|g |\overline{\partial} f_1|\right\|^{\frac{2p}{2+p}}
	&\leq C \int_{\Cn} |g e^{-\varphi}|^{\frac{2p}{2+p}} M_{2, r}(|\overline \partial f_1|)^{\frac{2p}{2+p}} dv \\
	&\leq C \|g\|^{\frac{2p}{2+p}} \left\{\int_{\Cn} M_{2, r}(|\overline \partial f_1|)^{p}dv\right\}^{\frac{2}{2+p}}.
\end{align*}
Similar estimate holds for $\phi=|f_2|$. Thus, $ M_\phi$ is bounded  from $F^2(\varphi)$ to $L^2(\varphi)$.
Now, for $g, h\in F^2(\varphi)$, it holds that
$$
	\langle M_\phi^* M_\phi g, h   \rangle = \langle M_\phi g, M_\phi h   \rangle
	= \int_{\Cn} g \overline h |\phi|^2 dv = \langle  T_{|\phi|^2} g, h\rangle,
$$
which in turn gives $M_\phi^* M_\phi= T_{|\phi|^2}$. Thus,  $ M_\phi \in S_{ p}$, and applying Theorem~\ref{IDA-space}, we get
$$
	\|M_\phi\|_{S_p}\le C \|M_{2, r}(\phi)\| _{L^p} \le C  \|f\|_{\textrm{IDA}^{p}}^p.
$$
This together with \eqref{kernel-image-d} give the norm estimate in \eqref{H-f1-e}.

\textbf{(A)$\Rightarrow$(B).} Suppose $f\in \mathcal S$ and  $H_f\in {S_p}(F^2(\varphi), L^2(\varphi))$. We will prove that
\begin{equation}\label{less-than-Sp-norm}
	\|f\|_{\textrm{IDA}^{p}}\le  C \|H_f\|_{S_p}.
\end{equation}
For this purpose, we borrow an idea from the proof of Proposition~6.8 in~\cite{FX18}. By Remark~\ref{independence}, it suffices to prove~\eqref{less-than-Sp-norm} for some $r\in (0, r_0)$, where $r_0$ is as in Lemma~\ref{basic-est}. To do this, let $\Lambda$ be an $r$-lattice as in \eqref{lattice-c}, and decompose $\Lambda=\cup_k \Lambda_k$ as in (\ref{Omega-docomposition}) with $K\ge 2\sqrt{n}$  so that $B(a , r)\cap B(b , r)= \emptyset$ if $a\neq b$ and
$a, b\in \Lambda _k$.  

We deal with the case $0<p\le 1$ first. Since
$$
	H_f\in {S_p}(F^2(\varphi), L^2(\varphi))\subset \mathcal B (F^2(\varphi), L^2(\varphi)),
$$
Theorem  \ref{Stroethoff} shows that $fk_a-P(fk_a) \in L^2_{\mathrm{loc}}$. Clearly $P(fk_a) \in H(\Cn)$, so
 $$
 fk_a\in L^2(B(a, r), e^{-2\varphi}dv)  \ \textrm{ and } \  P_{a, r}(fk_a)\in A^2(B(a, r), e^{-2\varphi}dv),
 $$
which implies that $\|\chi_{B(a, r)} fk_{a}-P_{a, r} (fk_{a})\|<\infty$.  Now for $a\in \Lambda_k$, set
\begin{equation*}
g_a= \begin{cases}
	\frac {\chi_{B(a , r)} fk_{a}-P_{a, r} (fk_{a })}{\|\chi_{B(a, r)} fk_{a}-P_{a, r} (fk_{a})\|}
	& \textrm { if } \|\chi_{B(a, r)} fk_{a}-P_{a, r} (fk_{a})\|\neq 0, \\
	0  & \textrm { if } \|\chi_{B(a, r)} fk_{a}-P_{a, r} (fk_{a})\|= 0.
\end{cases}
\end{equation*}
It is easy to see that $\|g_a\|\le 1$ and $\langle g_{a}, g_{b} \rangle=0 $ if $B(a , r)\cap B(b , r)= \emptyset$. Let $J$ be any finite sub-collection of $\Lambda_k$,
and let $\{e_a\}_{a\in J} $ be an orthonormal set of $L^2(\varphi)$. Define
$$
	A= \sum_{a\in J} e_a\otimes g_a:  \ L^2(\varphi)\to L^2(\varphi).
$$
It is trivial to see that $A$ is of finite rank and
\begin{equation}\label{A-operator}
	\|A\|_{L^2(\varphi)\to L^2(\varphi)}\le  1.
\end{equation}

We now define another operator $T: L^2(\varphi)\to F^2(\varphi)$ by
$$
	T= \sum_{a\in J}  k_{a}\otimes e_a.
$$
Since  $\Lambda$ is separated, by Lemma \ref{tensor-product}, there is a constant $C$ depending only on $n$ and $r$ such that
\begin{equation}\label{bounded-T}
	\|T\|_{L^2(\varphi)\to F^2(\varphi)}\le C.
\end{equation}
It is easy to verify  that
\begin{equation}\label{AHT-a}
 	AH_fT= \sum_{a, \tau \in J} \left \langle  H_f k_{ \tau}, g_a \right \rangle e_a\otimes e_\tau
	= Y+Z,
\end{equation}
where
\begin{equation}\label{AHT}
	Y =  \sum_{a\in J } \left \langle  H_f k_{a}, g_a \right\rangle e_a\otimes e_a,\ \
	Z = \sum_{a, \tau \in J, \, a\neq\tau} \left \langle  H_f k_{ \tau}, g_a\right \rangle e_a\otimes e_\tau.
\end{equation}
By Lemma \ref{local-projection} and Lemma \ref{basic-est},
\begin{align*}
	\left  \langle  H_f k_{a }, g_a \right \rangle
	&=\left \langle   f k_{a }-P(f k_{a}),   g_a \right \rangle\\
	&=\left \langle \chi_{B(a  , r)}   f k_{a }-P_{a , r}(f k_{a }), g_a \right \rangle\\
	&=\left \|  \chi_{B(a, r)}   f k_{a }-P_{a , r}(f k_{a }) \right\|\\
	&\ge C \left \|   f-\frac 1 {  k_{a } }P_{a , r}(f k_{a }) \right\|_{L^2(B(a , r), dv)}.
\end{align*}
Further, by definition,
$$
	\left  \langle  H_f k_{a }, g_a \right \rangle \ge C \left \|   f-\frac 1 {  k_{a } }P_{a , r}(f k_{a })\right\|_{L^2(B(a , r), dv)} \ge C G_{r}(f)( a ).
$$
Thus, there is an $N$ independent of $f$ and $J$ such that
\begin{equation}\label{Y-est-a}
	\|Y\|_{S_p} ^p
	=\sum_{ a  \in J }\left \langle  H_f k_{a }, g_a \right\rangle^p
	\ge  N \sum_{ a \in J}  G_{r}(f)( a )^p.
\end{equation}
On the other hand, for $0<p\le 1$, applying Lemma 5 of \cite{Lu87} gives
\begin{equation}\label{Z-est}
	\left\|Z\right\|_{S_p} ^p
	\le  \sum_{a, \tau \in J, \, a\neq \tau} \left| \left \langle  H_f k_{ \tau}, g_a \right\rangle \right|^p.
\end{equation}
Let $Q_{a, r}$ be the Bergman projection of $L^2(B(a, r), dv)$ onto the Bergman space $A^2(B(a, r), dv)$. Then
$$
	k_\tau Q_{a, r}(f)\in A^2(B(a, r), dv)= A^2(B(a, r), e^{-2\varphi} dv),
$$
and further $fk_\tau -P_{a, r}(fk_\tau)$ and $P_{a, r}(fk_\tau)- k_\tau Q_{a, r}f $ are orthogonal in the space $L^2(B(a, r), e^{-2\varphi} dv)$.  Thus, for    $a, \tau\in \Cn$,   by Parseval's identity, we get
$$
	\left\|  fk_\tau -P_{a, r}(fk_\tau)\right\|_{L^2(B(a, r), e^{-2\varphi} dv)}
	\le \left\|  fk_\tau -k_\tau Q_{a, r}f\right\|_{L^2(B(a, r), e^{-2\varphi} dv)}.
$$
Hence, by Lemma \ref{local-projection},
\begin{align*}
	\left| \left \langle  H_f k_{ \tau}, g_a \right\rangle \right|
	&=\left| \left \langle   f k_{\tau} -P \left( f k_{ \tau}  \right), g_a  \right\rangle \right|\\
	&=\left| \left \langle  \chi_{B(a, r)} f k_{ \tau} -P_{a , r} \left( f k_{ \tau}\right), g_a \right\rangle \right|\\
	&\le\left\|  f k_{ \tau}-P_{a, r}  \left( f k_{ \tau} \right) \right\|_{L^2(B(a, r), e^{-2\varphi} dv)}\\
	&\le\left\|  f k_{ \tau}- k_{ \tau} Q_{a, r}  \left(f\right) \right\|_{L^2(B(a, r), e^{-2\varphi} dv)}\\
	&\le\sup_{\xi\in B(a, r)} \left|k_{ \tau}(\xi)e^{-\varphi}\right |\, \left\|  f-  Q_{a , r}  \left( f  \right) \right\|_{L^2(B(a, r), dv)}\\
	&\le C e^{-\theta | a -  \tau|} \left\| f- Q_{a , r}  \left( f  \right) \right\|_{L^2(  B(a, r), dv)}.
\end{align*}
Notice also that
$$
	\left\|  f - Q_{a , r}  \left( f  \right) \right\|_{L^2(  B(a  , r), dv)} = G_{r}(f)(a).
$$
Since  $B  (  {\tau} , r)\cap B  ( a , r)= \emptyset$ for $a, \tau\in J$ with $a\neq \tau$,
\begin{align*}
	\sum_{\tau \in J, \, \tau \neq a }e^{-\frac {p\theta} 2| a - \tau|}
	&\le C \sum_{\tau \in J, \, \tau \neq a } \int_{B  (  {\tau} , r)}   e^{-\frac {p\theta} 2| a - \xi|} dv(\xi)\\
	&\le C \int_{\Cn}e^{-\frac {p\theta} 2| \xi|} dv(\xi) =C.
\end{align*}
Therefore, by \eqref{Omega-docomposition} and \eqref{k-z-est},
\begin{align*}
	\left\|Z\right\|_{S_p} ^p
	&\le\sum_{a, \tau \in J, \, a\neq \tau}   e^{-p \theta| a -  \tau|} G_{r}(f)(a)^p\\
	&\le\sum_{a \in J }    G_{r}(f)(a)^p \sum_{\tau \in J, \, \tau \neq a }e^{-p \theta| a- \tau|}\\
	&\le e^{-\frac {p\theta}{2 \sqrt{n}} Kr} \sum_{a\in J}     G_{r}(f)(a)^p \sum_{\tau\in J, \, \tau\neq a  }e^{-\frac{p\theta} 2| a - \tau|}\\
	&\le C  e^{-\frac{p\theta}{2 \sqrt{n}} Kr} \sum_{a\in J}    \,  G_{r}(f)(a)^p,
\end{align*}
and hence, we can pick some $K$ sufficiently large so that
\begin{equation}\label{Z-est-z}
\left \|Z\right\|_{S_p} ^p \le \frac {N} 4 \sum_{ a\in J}  G_{r}(f)( a )^p.
\end{equation}
Using the estimate
$$
	\left\| Y\right \|_{S_p}^p\le 2 \left \| AH_fT\right\|_{S_p}^p+  2 \left\| Z\right\|_{S_p}^p
$$
(see (6.9) in \cite{FX18} for example), we see that
$$
	N \sum_{ a\in J}  G_{r}(f)( a )^p\le 2 \left \| AH_fT\right \|_{S_p}^p + \frac {N} 2  \sum_{a\in J}  G_{r}(f)( a )^p.
$$
Since $J$  is finite,  we have
\begin{equation}\label{AHf-T-norm}
	{N}  \sum_{j}  G_{r}(f)( a )^p\le 4 \left\| AH_fT\right\|_{S_p}^p,
\end{equation}
which can be further estimated, using \eqref{bounded-T}, as follows
\begin{equation}\label{aaa}
	\left\| AH_fT\right\|_{S_p}^p
	\le  \| A\|_{L^2(\varphi)\to L^2(\varphi)}^p \| H_f  \|_{S_p}^p  \|T\|_{L^2(\varphi)\to F^2(\varphi)}^p
	\le  C\left \| H_f \right \|_{S_p}^p.
\end{equation}
Puting (\ref{AHf-T-norm}) and (\ref{aaa}) together and applying the duality between $l^1$ and $l^\infty$, we obtain
$$
	\sum_{ a \in J}  G_{r}(f)( a )^p \le C \left\| H_f \right \|_{S_p}^p.
$$
The constants $C$ above are all independent of $f$ and $J$. Hence,
\begin{equation}\label{G-f-est-a}
	\sum_{ a\in \Lambda_k }  G_{r}(f)( a)^p \le C \left\| H_f \right \|_{S_p}^p.
\end{equation}
Now take $\Lambda$ to be an $\frac r {2}$-lattice similar to \eqref{lattice-c},  which can be viewed as a union of $4^n$ $r$-lattice. Then
\begin{align*}
	\int_{\Cn} G_{\frac r2}(f)^p dv
	&\le \sum_{ a\in \Lambda } \int_{B(a, \frac r2)} G_{\frac r2}(f)^p dv\\
	&\le C  \sum_{ a\in \Lambda } \sup_{z\in B(a, \frac r2)} G_{\frac r2}(f)(z)^p\\
	&\le C  \sum_{ a\in \Lambda }G_{r}(f)(a)^p\le  C\left \| H_f \right \|_{S_p}^p,
\end{align*}
and so, for $0<r\le r_0$, we have
$$
	\int_{\Cn} G_{\frac r2}(f)^p dv \le  C \left \| H_f \right \|_{S_p}^p.
$$
Therefore, by Theorem~\ref{IDA-space}, for each $r >0$, it holds that
\begin{equation}\label{G-f-est}
	\int_{\Cn} G_{r}(f)^p dv \le C  \left \| H_f \right \|_{S_p}^p.
\end{equation}

Now we treat the case $1\le p<\infty$. Let $\{e_a: a\in \Lambda_k\}$ be an orthonormal basis of $F^2(\varphi)$ and define linear operators  $T$ and $B$ by setting
$$
	T= \sum_{a\in \Lambda}  k_{a}\otimes e_a:\ L^2(\varphi)\to F^2(\varphi),
$$
and
$$
	B= \sum_{a\in \Lambda}  g_{a}\otimes e_a:\ L^2(\varphi)\to L^2(\varphi)
$$
where
$$
	g_a=\begin{cases}
	\frac {\chi_{B(a , r)} H_f(k_{a })}{ \| \chi_{B(a, r)} H_f(k_{a})\|} & \textrm { if } \| \chi_{B(a, r)} H_f(k_{a})\|\neq 0 \\
	0, & \textrm { if } \| \chi_{B(a, r)} H_f(k_{a})\|= 0.
	\end{cases}
$$
Since $\|g_a\|\le 1$ and $ \langle g_a, g_\tau \rangle=0$ when $a\neq \tau$, it follows that
$$
	\|B\|_{L^2(\varphi)\to L^2(\varphi)} \le 1.
$$
For $H_f\in S_p$, by Theorem \ref{Stroethoff}, we have $\lim_{z\to \infty}
\| \chi_{B(z, r)} H_f(k_{z})\|=0$.
Since
$$
	\left\langle  B^* M_{\chi_{B(a, r)}} H_f T e_a, e_a  \right \rangle
	= \left\langle    \chi_{B(a, r)} H_fT(e_{a}), B(e_{a}) \right \rangle
	= \| \chi_{B(a, r)} H_f(k_{a})\|,
$$
and
$$
	\left\langle  B^* M_{\chi_{B(a, r)}} H_f T e_a, e_b  \right \rangle =0 \ \ \textrm{for}\ a\neq b,
$$
$ B^* M_{\chi_{B(a, r)}} H_f T$ is a compact  positive  operator on $L^2(\varphi)$. Theorem 1.27 of \cite{Zh07} yields
$$
	\sum_{a\in \Lambda_k} \left|\left \langle  B^* M_{\chi_{B(a, r)}} H_f T e_a, e_a \right\rangle \right|^p \le \left\|  B^* M_{\chi_{B(a, r)}} H_f T \right\|_{S_p}^p \le C \left \|    H_f  \right\|_{S_p}^p.
$$
Thus, using (\ref{key-estimate}), we have
\begin{align*}
	\sum_{a\in \Lambda_k}   G_{r}(f)(a)^p
	&\le  C \sum_{a\in \Lambda_k}  \| \chi_{B(a, r)} H_f(k_{a})\|^p \\
	&=\sum_{a\in \Lambda_k} \left|\left \langle  B^*M_{\chi_{B(a, r)}} H_f T e_a, e_a \right\rangle \right|^p \le C \left \|    H_f  \right\|_{S_p}^p
\end{align*}
which gives~\eqref{G-f-est-a} for $1\le p<\infty$. From this, with the same approach as in the other case, we obtain the desired conclusion in \eqref{G-f-est}.

\textbf{(B)$\Rightarrow$(C).}
Suppose $f\in \mathrm{IDA}^{p}$, and decompose $f=f_1+f_2$ as in the implication (B)$\Rightarrow$(A).

For  $0<p\le 2$,  applying (\ref{f_1-estimate}) and Lemma \ref{integral-est} for $d\mu=  \left|{\overline \partial} f_1\right|^2 dv$, we get
\begin{align*} \label{kernel-image-a}
\begin{split}
	&\int_{\Cn} \left\|H_{f_1}(k_z)\right\|^p dv \\
	&\le C \int_{\Cn} \left (\int_{\Cn} \left |k_z(\xi) e^{-\varphi(\xi)}\right| ^2  \left|{\overline \partial} f_1(\xi)\right|^2 dv(\xi)\right)^{\frac p 2} dv(z)\\
	&\le  C \int_{\Cn}  dv(z) \int_{\Cn} \left |k_z(\xi) e^{-\varphi(\xi)}\right| ^p  M_{2, r }(\overline \partial f_1)(\xi) ^p dv(\xi) \\
	&\le  C  \int_{\Cn} M_{2, r }(\overline \partial f_1)(\xi) ^p dv(\xi).
\end{split}
\end{align*}
This gives
\begin{equation} \label{Schatten-p-k}
 \int_{\Cn} \left\|H_{f_1}(k_z)\right\|^p dv \le C \int_{\Cn} M_{2, r }(\overline \partial f_1)(\xi) ^p dv(\xi)\le C  \left\|   f \right\|_{\mathrm{IDA}^{p}}^p.
\end{equation}
Similarly,
\begin{align*}
	\|  H_{f_2}(k_z)\|^p
  	&\le C \left(\int_{\Cn} \left|k_z e^{-\varphi} \right|^2 M_{2, r}(f_2)^2 dv \right)^{\frac p 2}\\
	&\le C \int_{\Cn} \left| k_z e^{-\varphi} \right|^p M_{2, 2r}(f_2)^p dv.
\end{align*}
Integrating both sides with respect to $z$ over $\Cn$, we get
\begin{equation} \label{Schatten-p-m}
	\int_{\Cn}  \|  H_{f_2}(k_z)\|^p dv(z)\le  C \int_{\Cn}   M_{2, 2r}(f_2)^p dv\le C  \left\|   f \right\|_{\mathrm{IDA}^{p}}^p.
\end{equation}

For $2\le p<\infty$, by \eqref{f_1-estimate},
\begin{align*}
	\|H_{f_1}(k_z)\|^p
	&\le C \left \langle  |\overline \partial f_1|^2 k_z, k_z  \right\rangle ^{\frac p 2}\\
	&= C \left \langle T_{ |\overline \partial f_1|^2} k_z, k_z  \right\rangle ^{\frac p 2} \le C \left \langle \left( T_{ |\overline \partial f_1|^2}\right )^{\frac p 2} k_z, k_z  \right\rangle.
\end{align*}
Therefore,
\begin{equation} \label{Schatten-p-n1}
\begin{split}
	\int_{\Cn} \|H_{f_1}(k_z)\|^p dv(z)&\le C \int_{\Cn} \left \langle \left( T_{ |\overline \partial f_1|^2}\right )^{\frac p 2} k_z, k_z  \right\rangle dv(z)\\
 	&\le C \left\| T_{ |\overline \partial f_1|^2}\right\|_{S^{\frac p 2}}^{\frac p 2} \le C\left \| M_{2,  p}( |\overline \partial f_1|) \right \|_{L^p}\\
	&\le C \left\|   f \right\|_{\mathrm{IDA}^{p}}^p.
\end{split}
\end{equation}
Similarly,
$$
	\|H_{f_2}(k_z)\|^p  \le \left \langle  |  f_2|^2 k_z, k_z  \right\rangle ^{\frac p 2}  \le  \left \langle \left( T_{ |  f_2|^2}\right )^{\frac p 2} k_z, k_z  \right\rangle,
$$
and so
\begin{equation}\label{Schatten-p-n2}
\begin{split}
	\int_{\Cn} \|H_{f_2}(k_z)\|^p dv(z) &\le \left\|\left( T_{ |  f_2|^2}\right )^{\frac p 2} \right\|_{S_1}\\
	&\le   C\left \| M_{2,  p}(  | f_2|) \right \|_{L^p} \le C \left\|   f \right\|_{\mathrm{IDA}^{p}}^p.
\end{split}
\end{equation}
Combining the estimates \eqref {Schatten-p-k}--\eqref{Schatten-p-n2} gives
\begin{equation}\label {Schatten-p-o}
\int_{\Cn} \|H_{f}(k_z)\|^p dv(z) \le C \left\|   f \right\|_{\mathrm{IDA}^{p}}^p.
\end{equation}

\textbf{(C)$\Rightarrow$(B).} By \eqref{key-estimate} and \eqref{norm-euival}, we have
\begin{equation}\label {Schatten-p-p}
	\left\|   f \right\|_{\mathrm{IDA}^{p}}^p\le C \int_{\Cn} \|H_{f}(k_z)\|^p dv(z).
\end{equation}

Finally, the $S_p$-norm equivalence in \eqref{bounded-g} follows from (\ref{Schatten-p-a}), (\ref{G-f-est}), (\ref{Schatten-p-o}) and (\ref{Schatten-p-p}). The proof is completed.
 \end{proof}

Similarly to Corollary \ref{Stroethoff-a}, restricting to the classical Segel-Bargmann space, we have the following corollary, which was obtained only for $p=2$ in \cite{Bauer}.

 \begin{corollary}\label{standard-Fock}
 Suppose $\varphi(z)=\frac \alpha 2 |z|^2$ and $0<p<\infty$. Then for $f\in \mathcal S$, $H_f\in S_p$ if and only if $\int_{\Cn}\left \|(I-P)\left(f\circ \tau_z\right)\right\|^p dv(z)<\infty$. Furthermore,
\begin{equation}\label{Schatten-p-r}
 	\left\|H_f\right\|_{S_p} \simeq \left\{  \int_{\Cn}\left \|(I-P)\left(f\circ \tau_z\right)\right\|^p dv(z)\right\}^{\frac 1p}.
\end{equation}
\end{corollary}

\section{Simultaneous membership of $H_f$ and $H_{\overline f}$ in $S_p$}\label{proof of Simul-S-p}

In the setting of the classical Segal-Bargmann space $F^2$, for $0<p<\infty$ and $f\in \mathcal S$, both $H_f$ and $H_{\bar f}$ are in $S_p$ if and only if
\begin{equation}\label{e:SD}
	\int_{\C^n} \left( SD(f\circ \tau_z)\right)^p dv(z)<\infty,
\end{equation}
where $\tau_z(w) = w+ z$ and the standard deviation $SD(g)$ of $g \in L^2(\C^n, d\mu)$ with $d\mu(z) = \pi^{-n} e^{-|z|^2}dv(z)$ is defined by
$$
	SD(g) = \left\{\int_{\C^n} \Big| g - \int_{\C^n} g\, d\mu \Big|^2 d\mu\right\}^{1/2}
	=\left\{\int_{\C^n} |g|^2\, d\mu - \Big| \int_{\C^n} g\, d\mu \Big|^2\right\}^{1/2}
$$
(see \cite{I2011, XZ2004}). In~\cite{F11}, the result for $1\le p<\infty$ was extended to symmetrically-normed ideals $S_{\Phi}$, that is, it was proved that $H_f, H_{\bar f}\in S_{\Phi}$ if and only if
\begin{equation}\label{e:SNI}
	\Phi\left( \{ J(f; u)\}_{u\in \Z^{2n}}\right) < \infty,
\end{equation}
where $\Phi$ is a symmetric norming function,
$$
	J(f; u) = \left\{ \int_{Q+u}\int_{Q+u} |f(z)-f(w)|^2 dv(w) dv(z)\right\}^{1/2},
$$
and $Q = \{ (x_1+iy_1, \ldots, x_n+iy_n) \in \C^n : x_1, y_1, \ldots x_n,y_n\in [-1, 2)\}$.

To compare the situation with weighted Segal-Bargmann spaces, we say that $f\in L^2_{\rm loc}$ is in $\IMO^{s}$ with $0<s\le\infty$ if $MO_{2, r} (f) \in L^s$ for some $r>0$, where
$$
	MO_{2,r}(f)(z) = \left( \frac1{|B(z,r)|} \int_{B(z,r)} |f - \widehat f_r(z)|^2\, dv\right)^{1/2}
$$
and $\widehat f_r$ is the average function defined on $\C^n$ by
$$
	\widehat f_r(z) = \frac1{|B(z,r)|} \int_{B(z,r)} f\, dv.
$$
For further details on these spaces, see~\cite{HW18}.

The following lemma shows the connection between $\IMO$ and $\IDA$.

\begin{lemma}\label{H-f-H-bar-f}
Suppose $ 0< p\leq \infty$. Then for  $f\in L^{2}_{{\rm loc}}(\Cn)$,  $f\in \mathrm{IDA}^{p}$ and $\overline f\in \mathrm{IDA}^{p}$  if and only if $f\in \mathrm{IMO}^{p}$. Furthermore,
$$
	\| f\|_{\mathrm{IDA}^{p}} +\|\overline f\|_{\mathrm{IDA}^{p}} \simeq \|f\|_{\mathrm {IMO}^{p}}.
$$
\end{lemma}

\begin{proof} The conclusion for $ 1< p\leq \infty$ is essentially  Proposition 2.5 in \cite{HW18}. As before, denote by $Q_{z, r}$ the Bergman projection of $L^2(B(z, r), dv)$ onto $A^2(B(z, r), dv)$.  If $f\in L^p_{\rm loc}(\Cn)$,  set $h_1= Q_{z, r}(f)$ and , $h_2= Q_{z, r}(\overline f)$. Then
$$
       \frac 1{|B(z, r)|}\int_{B(z, r)}\left| f_j-h_j \right|^2 dv = G_{r}(f_j)(z)^2,
$$
where $f_1 = f$ and $f_2=\bar f$. Set
$
	c(z)= \frac12{\rm  Re}(h_{1}+h_{2})(z)+\frac12{\rm i}\, {\rm  Im} (h_{1}-h_{2})(z).
$
As shown in the proof of Proposition 2.5 of \cite{HW18},
$$
	\left\{   \frac 1{|B(z, r)|}\int_{B(z, r)}\left|  f-c(z) \right|^2 dv \right\}^{\frac 12} \le C\left\{ G_{r}(f)(z)  +  G_{r}(\overline f)(z) \right\}.
$$
Hence,
\begin{align*}
	\frac 1{|B(z, r)|}\int_{B(z, r)}\left|  f-\frac 1{|B(z, r)|}\int_{B(z, r)} fdv \right|^2 dv
	&\le  \frac 1{|B(z, r)|}\int_{B(z, r)}\left|  f-c(z) \right|^2 dv\\
	&\le C \left( G_{r}(f)(z) + G_{r}(\overline f)(z) \right)^2.
\end{align*}
This implies, for $0<p\le \infty$,
$
	\|f\|_{\mathrm{IMO}^{p}}\le C\left\{  \|f\|_{\mathrm{IDA}^{p}}+  \|\overline f\|_{\mathrm{IDA}^{p}} \right\}.
$

The reverse inequality follows from the fact that  $ \|f\|_{\mathrm{IDA}^{p}}\le  \|f\|_{\mathrm{IMO}^{p}}$ since $ \|\overline f\|_{\mathrm{IDA}^{p}}\le  \|f\|_{\mathrm{IMO}^{p}}$ by definition.
\end{proof}

\begin{theorem} \label{Simul-S-p}
Let $0<p<\infty$ and suppose $\varphi\in C^2(\Cn)$ is real valued with   $\mathrm{i} \partial \overline{\partial} \varphi \simeq \omega_0$. Then for  $f\in \mathcal S$, the following statements are equivalent.
\begin{itemize}
\item[(A)] Both $H_f, H_{\overline f} \in {S_p}(F^2(\varphi), L^2(\varphi))$.

\item[(B)] $f\in \mathrm{IMO}^{p}$.
\end{itemize}
Furthermore,
\begin{equation}\label{simul-S-p-a}
	\|H_f\|_{S_p} +  \|H_{\overline f} \|_{S_p}\simeq \left\|   f \right\|_{\mathrm{IMO}^{p}}.
\end{equation}
\end{theorem}

\begin{proof}
Given $f\in \mathcal S$ and $0<p<\infty$, the equivalence between (A) and (B) together with the norm estimates \eqref{simul-S-p-a} follow from Theorem \ref{S-p-criteria} and Lemma~\ref{H-f-H-bar-f}.
\end{proof}

For Hankel operators acting on $F^2$, the conditions in \eqref{e:SD} and \eqref{e:SNI} are of course equivalent to (B), but notice that the latter equivalence can be proved directly.

\section{Proof of Theorem~\ref{Berger-Coburn Schatten}}\label{Berger-Coburn phenomenon}
In this section we prove the Berger-Coburn phenomenon for Schatten $p$-class Hankel operators when $1<p<\infty$. For this purpose we employ the Ahlfors-Beurling operator which is a well-known Calder\'{o}n-Zygmund operator on $L^p(\mathbb C)$, $1<p<\infty$, defined as follows
$$
	\mathfrak{T} (f)(z)= p. v.\, - \frac 1\pi \int_{\mathbb C} \frac {f(\xi)}{(\xi-z)^2} dv(\xi),
$$
where $p. v.$ means the Cauchy principal value. See \cite{Ah06} and \cite{AIM09} for further details.

\begin{lemma} \label{partial-derivatives} Suppose $1<p<\infty$. Then there is a constant $C$ depending only on $p$ such that, for $f\in C^2(\Cn)\cap L^\infty$ and $j=1, 2, \cdots, n$,
\begin{equation}\label{partial-dir}
	\left\| \frac {\partial f}{\partial z_j} \right\|_{L^p}
	\le C  \left\| \frac {\partial f}{\partial \overline z_j} \right\|_{L^p}.
\end{equation}
\end{lemma}

\begin{proof}
We take $n=1$ temporarily. Let $f\in C^2(\mathbb C)\cap L^\infty$. If $\left\| \frac {\partial f}{\partial \overline z_j} \right\|_{L^p}=0$, then $f\in H(\mathbb C)\cap L^\infty$, which implies $f$ is constant and the estimates   in \eqref{partial-dir} follow. So we suppose $\left\| \frac {\partial f}{\partial \overline z_j} \right\|_{L^p}>0$. Take $\psi(r)\in C^\infty(\mathbb R)$ to be decreasing such that $\psi(x)=1$ for $x\le 0$, $\psi(x)=0$ for $x\ge 1$, and $0\le -\psi'(x)\le 2$ for $x\in  \mathbb R$.
For $R>0$ fixed, set $\psi_R(x)= \psi(x-R)$. Now for $f\in C^2(\Cn)\cap L^\infty$, define
$ f_R(z)=  f(z) \psi_R(|z|)$. It is trivial that $ f_R(z) \in C^2_c(\mathbb C)$,  the set of $C^2$ functions on $\mathbb R^2$   with compact support. From Theorem 2.1.1 in \cite{CS} we have
$$
	f_R (z)= \frac 1{2\pi \mathrm i} \int_{\mathbb C} \frac {\frac {\partial f_R}{\partial \overline z}}{\xi-z} d\xi\wedge d\overline \xi.
$$
Notice that $ \frac {\partial f_R}{\partial \overline z} = \psi_R \frac {\partial  f}{\partial \overline z}+ f \frac {\partial  \psi_R}{\partial \overline z} $.   By Lemma 2 on page 52 in~\cite{Ah06}, we get
\begin{equation}\label{partial-estimate-z}
	\frac {\partial f_R}{\partial z}(z)
	= \mathfrak{T}\left( \frac {\partial f_R}{\partial \overline z}\right)(z)
	= \mathfrak{T}\left(\psi_R \frac {\partial  f}{\partial \overline z}\right)(z)+ \mathfrak{T}\left(f \frac {\partial  \psi_R}{\partial \overline z}\right)(z).
\end{equation}
Now for $r>0$ and $|z|<r$, when $R$ is sufficiently large, it holds that
$$
	\mathfrak{T}\left(f \frac {\partial  \psi_R}{\partial \overline z}\right)(z) \le \frac{\|f\|_{L^\infty} } {\pi(R-r)^2}  \int_{R\le |\xi|\le R+1}   dv(\xi)\le  \frac{3R \|f\|_{L^\infty} } { (R-r)^2},
$$
and hence
\begin{equation}\label{partial-deri-a}
	\left\|\mathfrak{T}\left(f \frac {\partial  \psi_R}{\partial \overline z}\right)\right\|_{L^p(D(0, r), dv)}
	\le  \left\| \frac {\partial f}{\partial \overline z} \right\|_{L^p}.
\end{equation}
On the other hand, by the boundeness of $\mathfrak{T}$ on $L^p$ (see for example, Theorem 4.5.3 in \cite{AIM09}, or
the estimate (11) on page 53 in \cite{Ah06}), we get
\begin{equation}\label{partial-deri-b}
	\left\| \mathfrak{T}\left(\psi_R \frac {\partial  f}{\partial \overline z}\right)\right\|_{L^p}
	\le C \left\|  \psi_R \frac {\partial  f}{\partial \overline z} \right\|_{L^p}
	\le C \left\|\frac {\partial  f}{\partial \overline z} \right\|_{L^p}.
\end{equation}
From (\ref{partial-estimate-z}), (\ref{partial-deri-a}) and (\ref{partial-deri-b}) we obtain
$$
	\left\|  \frac {\partial f}{\partial z} \right\|_{L^p(D(0, r), dv)}
	= \left\|  \frac {\partial f_R}{\partial z} \right\|_{L^p(D(0, r), dv)}
	\le C \left\|\frac {\partial  f}{\partial \overline z} \right\|_{L^p}.
$$
Therefore,
\begin{equation}\label{derivative-est}
	\left\|  \frac {\partial f}{\partial z} \right\|_{L^p}
	\le C \left\|\frac {\partial  f}{\partial \overline z} \right\|_{L^p}.
\end{equation}

Now for $n\ge 2$ and $f\in L^\infty\cap C^2(\Cn)$, from (\ref{derivative-est})
\begin{align*}
	\int_{\Cn} \left|\frac {\partial f}{\partial z_1}(\xi)\right|^p dv(\xi)
	&=\int_{{\mathbb C}^{n-1}} dv(\xi') \int_{\mathbb C} \left|  \frac {\partial f}{\partial z_1}(\xi_1, \xi')\right|^p dv(\xi_1) \\
	&\le C  \int_{{\mathbb C}^{n-1}} dv(\xi') \int_{\mathbb C} \left| \frac {\partial f}{\partial \overline z_1}(\xi_1, \xi')\right|^p dv(\xi_1).
\end{align*}
This implies (\ref{partial-dir}) for $j=1$.  Similarly, we have (\ref{partial-dir}) for $j=2, \ldots, n$, which completes the proof.
\end{proof}

\begin{proof}[Proof of Theorem~\ref{Berger-Coburn Schatten}]
Suppose $1<p<\infty$ and $H_f\in S_p$. By Theorem \ref {S-p-criteria}, we have
$$
 	\|f\|_{ \mathrm{IDA}^{p}} \simeq \|H_f\|_{S_p}<\infty.
$$
We decompose $f=f_1+f_2$ as in \eqref{decomp-a}. Then, since $M_{2, r}(\overline {f_2}) = M_{2, r}( {f_2})\in L^p$, we have $H_{\overline f_2}\in S_p$ and
\begin{equation}\label{f-2-Sp-norm}
\|H_{\overline f_2} \|_{S_p} \le C \| M_{2, r}( {f_2})\|_{L^p}\le C \|f\|_{\mathrm{IDA}^{p}}.
\end{equation}
In addition, since $f\in L^\infty$, as in (5.3) of~\cite{HV20}, we may assume
$$
	\|f_1\|_{ L^\infty}\le C \|f\|_{ L^\infty},
$$
where the constant $C$ is independent of $f$. We now apply Lemma~\ref{partial-derivatives} to obtain
$$
	\left\| \frac {\partial \overline {f}_1}{\partial \overline {z}_j} \right\|_{L^p}
	=\left\| \frac {\partial f_1}{\partial z_j} \right\|_{L^p}
	\le C  \left\| \frac {\partial f_1}{\partial \overline z_j} \right\|_{L^p}.
$$
This and (\ref{H-f1-e}) yield
\begin{equation}\label{f-2-Sp-norm-m}
 \|H_{{\overline f}_1 }\|_{S_p}\le C \|\overline \partial\, \overline f_1 \|_{L^p}\le C  \|\overline \partial f_1 \|_{L^p}\le C \|f\|_{\mathrm{IDA}^{p}}.
\end{equation}
It follows from (\ref{f-2-Sp-norm}), (\ref{f-2-Sp-norm-m}) and Theorem \ref {S-p-criteria} that
$$
	\|H_{\overline f}\|_{S_p}\le C \|f\|_{\mathrm{IDA}^{p}} \le C   \|H_{f}\|_{S_p},
$$
which completes the proof.
\end{proof}

\section*{Acknowledgments}

\noindent The authors thank the referees for many useful suggestions.

\end{document}


\author{Zhangjian Hu}
\address{Department of Mathematics, Huzhou University, Huzhou, Zhejiang, China}
\email{huzj@zjhu.edu.cn}

\author{Jani A. Virtanen}
\address{Department of Mathematics and Statistics, University of Reading, Reading, England}
\email{j.a.virtanen@reading.ac.uk}
\address{Department of Mathematics and Statistics, University of Helsinki, Helsinki, Finland}
\email{jani.virtanen@helsinki.fi}

\title[Corrigendum]{Corrigendum to ``Schatten class Hankel operators on the Segal-Bargmann space and the Berger-Coburn phenomenon''}



\thanks{Z.~Hu was supported in part by the National Natural Science Foundation of China (12071130, 12171150) and J.~Virtanen was supported in part by Engineering and Physical Sciences Research Council grant EP/T008636/1.}

\begin{abstract}
The authors provide a correct proof of Theorem~1.2 and correct the statement of Theorem 2.6 in their paper, which appeared in Trans. Amer. Math. Soc. 375 (2022), 3733--3753.
\end{abstract}

\maketitle

Let $\varphi\in C^2(\C^n)$ be a real-valued weight that satisfies the property that there are two positive constants $m$ and $M$ such that $m\omega_0 \le \mathrm{i} \partial \overline{\partial} \varphi \le M\omega_0$ in the sense of currents,  where $\omega_0 = \mathrm{i} \partial \overline{\partial}  |z|^2 $ is the Euclidean-K\"{a}hler form on $\C^n$. The weighted Segal-Bargmann space $F^2(\varphi)$ is defined by $F^2(\varphi) = H(\C^n) \cap L^2(\varphi)$, where $H(\Omega )$  stands for all holomorphic functions on  $\Omega\subseteq \C^n$, $L^2(\varphi)=L^2(\C^n, e^{-2\varphi}dv)$, and $dv$ is the Lebesgue measure on $\C^n$. For $f\in L^\infty = L^\infty(\C^n, dv)$, the Hankel operator $H_f : F^2(\varphi) \to L^2(\varphi)$ is defined by
$$
	H_f g = (I-P)(fg),
$$
where $P$ is the Bergman projection of $L^2(\varphi)$ onto $F^2(\varphi)$.

The purpose of this corrigendum is to provide a correct proof of Theorem 1.2 in \cite{Schatten}, which establishes the Berger-Coburn phenomenon for the Schatten classes. The original proof is not complete because the first inequality in (7.7) of \cite{Schatten} is known to be true only when $2\le p<\infty$.

\begin{theorem1.2}\label{1.2}
Suppose $\varphi\in C^2(\C^n)$ is  real valued,  $\mathrm{i} \partial \overline{\partial} \varphi \simeq \omega_0$, and $1< p<\infty$.   Then for $f\in L^\infty$,  $H_f\in {S_p}$ if and only if $H_{\bar f} \in {S_p}$ with the $S_p$-norm estimate
\begin{equation}\tag{1.5}\label{S-p-norm}
	\|H_{\bar f}\|_{S_p}\le C  \|H_{ f }\|_{S_p},
\end{equation}
where the constant $C$ is independent of $f$.
\end{theorem1.2}

In addition, we correct (2.10) in Theorem 2.6 of \cite{Schatten} as follows.

\begin{theorem2.6}\label{2.6}
Suppose $0< s\le \infty$, and $f\in L^2_{\mathrm {loc}}$. Then $f\in  \mathrm{IDA}^{s}$ if and only if $f$ admits a decomposition $f=f_1+f_2$ such that
\begin{equation}\tag{2.9}\label{q-less-then-p-m}
	f_1\in C^2(\Cn), \, \, |\overline \partial f_1|+  M_{2,r}(\overline \partial f_1)  + M_{2, r}(f_2) \in L^{s}
\end{equation}
for some (or any) $r>0$. Furthermore,
\begin{equation}\tag{2.10}\label{norm-euival}
	\|f\|_{\IDA^{s}} \simeq \inf \{ \|M_{2,r}(\bar \partial f_1)\|_{L^s} + \|M_{2,r}(f_2)\|_{L^s}\}
\end{equation}
where the infimum is taken over all possible decompositions $f=f_1+f_2$ that satisfy  \eqref{q-less-then-p-m} with a fixed $r$.
\end{theorem2.6}

We start with a correct proof of Theorem 1.2 that works for all $1<p<\infty$, and then discuss the original proof and (2.10) briefly in Remarks \ref{R1} and \ref{R2}.

\begin{proof}[Proof of Theorem 1.2]
Suppose $H_f\in S_p$ with $1<p<\infty$. By Theorem 1.1 of \cite{Schatten},
$$
 	\|f\|_{ \mathrm{IDA}^{p}} \simeq \|H_f\|_{S_p}<\infty,
$$
where $\|f\|_{\IDA^p} = \|G_1(f)\|_{L^p} \simeq \|G_r(f)\|_{L^p}$ for $r>0$ and
$$
	G_{r}(f)(z)=\inf_{h\in H(B(z,r))} \left(\frac{1}{|B(z,r)|}\int_{B(z,r)} |f-h|^2 dv \right)^{\frac{1}{2}}\quad (z\in \C^n).
$$
Let $\{a_j\}$ be a fixed $\frac {r}{2}$-lattice, and choose a partition of unity $\{\psi_j\}_{j=1}^\infty$ subordinate to $\{B(a_j, r/2)\}$. Define the $q$th mean of $|f|$ over $B(z, r)$ by
$$
	M_{q, r}(f)(z)= \left(\frac{1}{|B(z,r)|}\int_{B(z,r)} |f|^q dv \right)^{\frac{1}{q}}.
$$
By Lemma~3.3 of~\cite{HV20}, there is an $h_j\in H(B(a_j, r))$ so that $\sup_{z\in B(a_j, r/2)}|h_j(z)|\le C \|f\|_{ L^\infty}$ and
$$
      M_{2, r}(f-h_j)(a_j)= G_{r}(f)(a_j).
$$
Decompose $f=f_1+f_2$ with
\begin{equation}\label{decomp-a}
	|\overline \partial f_1(z)|+ M_{2, r}(|\overline \partial f_1|)(z)  + M_{2, r}(f_2)(z) \le C G_{2r}(f)(z), \ z\in \C^n
\end{equation}
(see (2.8) of~\cite{Schatten}), where we may choose
$$
	f_1=\sum_{j=1}^\infty h_j\psi_j
$$
as in the proof of Lemma 3.6 of \cite{HV20}. Notice that $f_1\in L^\infty$ and $\overline {\partial}\, \overline{f}_1= F+H$, where
$$
	F= \sum_{j=1}^\infty  \overline {h}_j  \overline {\partial}\psi_j
	\quad{\rm and}\quad
	H= \sum_{j=1}^\infty  \psi_j \overline {\partial}\,\overline  {h}_j.
$$
With the same proof as the one for the estimate of $|\overline \partial  f_1|$ in \eqref{decomp-a}, we have
\begin{equation}\label{est-F}
	|F(z)|\le C G_{2r}(f)(z).
\end{equation}
Then Lemma 7.1 of \cite{Schatten} implies that
\begin{equation}\label{f-Sp-norm-b}
	\left \| H \right \|_{L^p} \le \|\overline \partial\, \overline f_1 \|_{L^p} + \|F \|_{L^p}
	\le C  \{\|\overline \partial f_1 \|_{L^p}+ \|F \|_{L^p}\}\le  C \|f\|_{\mathrm{IDA}^{p}}.
\end{equation}

For $z\in B(a_j, 3r/4)\cap B(a_k, 3r/4)$,  the coefficients $\overline{ \frac{ \partial(h_k-h_j)} {\partial{z}_l}}$ of $ \overline {\partial}\left(\overline  {h}_k-\overline  {h}_j\right)$ are conjugate holomorphic on $B(z, r/4)$. Applying the Cauchy estimate to each of the coefficients,
we obtain
\begin{align*}
	\left | \overline {\partial}\left(\overline  {h}_k(z) - \overline  {h}_j(z) \right) \right|
	&\le C \left\{ \int_{B(z, r/4)} \left | \overline  {h}_k(w) -\overline  {h}_j(w)   \right|^2 dv(w)\right\}^{\frac 12}\\
	&\le C  \left\{ G_{r}(f)(a_k)+   G_{r}(f)(a_j)\right\} \le C   G_{2r}(f)(z).
\end{align*}
Now, for $z\in   B(a_k, 3r/4)$, using the identity  $\overline {\partial}\,\overline  {h}_k  = \sum_{j=1}^\infty \psi_j \overline {\partial}(\overline  {h}_k  - \overline  {h}_j ) + H(z)$, we have
\begin{align*}
	\left | \overline {\partial}\,\overline  {h}_k(z) \right|^p
	&\le C \left\{ \left |\sum_{j=1}^\infty \psi_j (z) \overline {\partial}\left(\overline  {h}_k(z) - \overline  {h}_j(z) \right) \right|^p + |H(z)|^p \right\}\\
	& \le C \left\{ \sum_{j,   |a_j-z|<\frac r 2} \psi_j (z)\left | \overline {\partial}\left(\overline  {h}_k(z) - \overline  {h}_j(z) \right) \right|^p + |H(z)|^p \right\}\\
	& \le C \left\{ G_{2r}(f)(z)^p + |H(z)|^p\right\}.
\end{align*}
For $z\in B(a_k, r/2)$, we have $B(z, r/4)\subset B(a_k, 3r/4)$, and it follows from the plurisubharmonicity that
\begin{equation}\label{e:PSH}
 |\overline {\partial}\,\overline  {h}_k (z)|^{p}\le  M_{p, r/4}( \overline {\partial}\,\overline  {h}_k)(z)^p \le C\left\{
     G_{3r}(f)(z)^p + M_{p, r/4}( H)(z)^p \right \},
\end{equation}
where the estimate $G_{2r}(f) (w)\le C G_{3r}(f) (z) $ for $w\in B(z, r/4)$ was also used for the latter inequality. Now for $z\in \mathbb C^n$,  there is some  $w'\in \overline{B(z, r/4)}$ so that
$$
	M_{2, r/4}(H)(z)^p \le \max\{ |H(w) |^p: w\in \overline{B(z, r/4)}\}\\
	=\left|\sum_{k=1}^\infty \psi_k(w')  \overline{\partial} \, \bar h_k(w')\right|^p.
$$
For  $w'\in \overline{B(z, r/4)}$, it holds that $G_{3r}(f)(w')^p \le C G_{4r}(f)(z)^p$ and   $M_{p, r/4}(H)(w')^p\le C M_{p, r/2}( H)(z)^p$. Thus, using \eqref{e:PSH}, we obtain
\begin{align*}
	M_{2, r/4}(H)(z)^p &\le C \sum_{k=1}^\infty \psi_k(w') \left| \overline{\partial} \, \bar h_k(w')\right|^p \\
	& \le C\sum_{k, \psi_k(w') \neq 0} \psi_k(w')\left\{G_{3r}(f)(w')^p + M_{p, r/4}( H)(w')^p \right \}\\
	& \le C\sum_{k, \psi_k(w') \neq 0} \psi_k(w')\left\{G_{4r}(f)(z)^p + M_{p, r/2}( H)(z)^p \right \}\\
	&= C \left\{G_{4r}(f)(z)^p + M_{p, r/2}( H)(z)^p \right \}.
\end{align*}
This, together with \eqref{est-F}, \eqref{f-Sp-norm-b} and $\int_{{\mathbb C}^n} M_{p, r/2}( H)^p dv = \int_{{\mathbb C}^n} |H|^p dv$, implies that
\begin{align*}
	\int_{{\mathbb C}^n} M_{2, r/4}(  \overline {\partial}\, \overline{f}_1 )^p dv
	&\le C \left\{  \int_{{\mathbb C}^n} M_{2, r/4}\left( H \right)^pdv +  \int_{{\mathbb C}^n} M_{2, r/4}( F)^pdv  \right\}\\
	&\le C\left\{ \int_{\mathbb C^n}\left[ G_{4r}(f)^p + M_{p, r/2}( H)^p \right] dv +  \int_{{\mathbb C}^n} M_{2, r/4}( F)^p dv \right\}\\
	&\le C \|f\|_{\mathrm{IDA}^{p}}^p <\infty.
 \end{align*}
It follows from Theorem 1.1 and the arguments on page 3748 of \cite{Schatten} that
\begin{equation}\label{estimate}
	\|H_{\bar f_1}\|_{S_p}^p \le C \int_{\C^n} M_{2,r/4}(\bar\partial \bar f_1)^p\,dv,
\end{equation}
and therefore,
\begin{equation}\label{e:f_1}
	\|H_{\bar f_1}\|_{S_p}\le C \|f\|_{\mathrm{IDA}^{p}} \le C   \|H_{f}\|_{S_p}.
\end{equation}

The Schatten norm of $H_{\bar f_2}$ can be dealt with as in the proof of Theorem~1.2 of~\cite{Schatten}. This completes the proof.
\end{proof}

\begin{remark}\label{R1}
The preceding proof concerning $H_{\overline f_1}$ simplifies when $2\le p<\infty$. Indeed, if $f=f_1+f_2$ as in (2.8) of \cite{Schatten} with $\|f_1\|_{L^\infty}<\infty$, then, similarly to \eqref{estimate},
$$
	\|H_{\bar f_1}\|_{S_p} \le C \|M_{2,r}(\bar\partial \bar f_1)\|_{L^p}.
$$
Therefore, when $2\le p<\infty$, so that $p/2\ge 1$, H\"older's inequality implies that
$$
	\|M_{2,r}(\bar\partial \bar f_1)\|_{L^p} \le \|\bar \partial \bar f_1\|_{L^p},
$$
and so by Lemma~7.1 of \cite{Schatten},
$$
	\|\bar\partial \bar f_1\|_{L^p} \le C\|\bar\partial f_1\|_{L^p}.
$$
Finally, it follows from (2.8) of \cite{Schatten} that
$$
	\|\bar \partial f_1\|_{L^p}\le C\|f\|_{\IDA^p}.
$$
The preceding inequalities show that (7.7) in \cite{Schatten} is actually true when $2\le p<\infty$.
\end{remark}

\begin{remark}\label{R2}
Whether the first inequality in (7.7) of the original proof in \cite{Schatten} is true when $1<p<2$ remains open. This is in fact related to the claim that
$$
	\|f\|_{\IDA^s} \simeq \inf \{ \|\bar\partial f_1\|_{L^s} + \|M_{2,r}(f_2)\|_{L^s}\},
$$
which appeared in (2.10) of \cite{Schatten} but remains unproven. As stated on the first page, a correct assertion is
$$
	\|f\|_{\IDA^{s}} \simeq \inf \{ \|M_{2,r}(\bar \partial f_1)\|_{L^s} + \|M_{2,r}(f_2)\|_{L^s}\},	
$$
which follows from Theorem 3.8 of \cite{HV20}.
\end{remark}